\documentclass{article}
\usepackage{amsmath}
\usepackage{amssymb}
\usepackage{amsfonts}
\usepackage{amsthm}
\usepackage{ascmac}
\usepackage{authblk}
\usepackage{indentfirst}

\title{Reciprocity law of finite Galois extension fields using Jacobian Varitey}
\author{Shinji Ishida\thanks{email address: nanacanji@gmail.com}}

\providecommand{\keywords}[1]{{\textit{Keywords:}} #1}

\makeatletter
\@addtoreset{equation}{section}

\newtheorem{thm}{Theorem}[section]

\newtheorem{coro}{Corollary}[section]

\begin{document}
\maketitle
%%%%%%%%%%%%%%%%%%%%%%%%%%%%%%%%%%%%%%%%%%%%%%%%%%%%%%%%%%%%%%%%%%%%%%%%%%%%%%%%Abstract starts
%%%%%%%%%%%%%%%%%%%%%%%%%%%%%%%%%%%%%%%%%%%%%%%%%%%%%%%%%%%%%%%%%%%%%%%%%%%%%%%

\begin{abstract}
We discuss a reciprocity law of finite Galois extension fields defied by odd degree irreducible polynomials $f(x) = x^{2g+1}+a_{1}x^{2g}+a_{2}x^{2g-1}+\cdots+a_{2g}x+a_{2g+1}$ over algebraic integer ring $\frak{O}_{K}$ of an algebraic field $K$ $([K:\mathbb{Q}]<\infty$). Let $\frak{p}$ be a prime ideal of $\frak{O}_{K}$ which is relatively prime to 2 and conductor $\mathfrak{F} =\{\alpha\in \mathfrak{O}_{K(\theta)}|$ $(\alpha)\mathfrak{O}_{K(\theta)}\subset \mathfrak{O}_{K}[\theta]$, where $f(\theta)=0$ and $\theta\in K_{f}\}$. Then our reciprocity law says that $\frak{p}$ splits completely over the minimal splitting field $K_{f}$ of $f(x)$ over $K$ if and only if $J_{f, p}(2)$ $\cong$ $\mathbb{Z}/2\mathbb{Z}^{\oplus 2g}$, where $J_{f,\frak{p}}$ is a Jacobian Variety of a projective non-singular irreducible model of $y^2=f(x)$ of genus $g$ over $\frak{O}_{K}/\frak{p}$ and $J_{f, \frak{p}}(2)$ is 2-torsion rational points of $J_{f, \frak{p}}$ over $\frak{O}_{K}/\frak{p}$. As an interesting corollary, Galois group $Gal(K_{f}/K)$ is a subgroup of general linear group $GL_{2g}(\mathbb{F}_{2})$.
\end{abstract}
\keywords\footnote[1]{The subject classification codes by 2020 Mathematics Subject Classification is primary-11R37.}{Algebraic number theory: global fields, nonabelian class field theory}

%%%%%%%%%%%%%%%%%%%%%%%%%%%%%%%%%%%%%%%%%%%%%%%%%%%%%%%%%%%%%%%%%%%%%%%%%%%%%%%%Introduction starts
%%%%%%%%%%%%%%%%%%%%%%%%%%%%%%%%%%%%%%%%%%%%%%%%%%%%%%%%%%%%%%%%%%%%%%%%%%%%%%%

\section{Introduction}
Let $f(x) = x^{2g+1}+a_{1}x^{2g}+a_{2}x^{2g-1}+\cdots+a_{2g}x+a_{2g+1}$ be an odd degree irreducible polynomial over algebraic integer ring $\frak{O}_{K}$ of an algebraic field $K$ $([K:\mathbb{Q}]<\infty$), and $K_{f}$ be the minimal splitting field of $f(x)$ over $K$. The purpose of this paper is to observe what kind of prime ideals of $\mathfrak{O}_{K}$ completely split over $K_{f}$ through Jacobian variety $J_{f}$ of a projective non-singular model of a hyper elliptic curve $y^2=f(x)$. We embed all roots of $f(x)=0$ into the 2-torsion points of Jacobian Variety and observe a prime ideal $\frak{p}$ of $K$ which is relatively prime to 2 and conductor $\mathfrak{F}$ splits completely over $K_{f}$ if and only if  $J_{f, p}(2)$ $\cong$ $\mathbb{Z}/2\mathbb{Z}^{\oplus 2g}$, where $J_{f, p}$ is a reduction of $J_{f}$ over $\frak{O}_{K}/\frak{p}$ and $g$ is genus of the projective non-singular model of hyper elliptic curve $y^2=f(x)$.\\

Remember usual cyclotomic field theory. According to usual cyclotomic field theory, for $n>2$, a prime number $p$ splits completely over $\mathbb{Q}(\zeta_{n})$ if and only if $p\equiv 1$ (mod $n$). This "$p\equiv 1$ (mod $n$)" means "all roots of $x^n=1$ over $\overline{\mathbb{F}}_{p}$ are included in rational points of commutative group variety $\mathbb{F}_{p}^{\ast}$". In other words, all roots of $x^n=1$ embedded into torsion subgroup of  a commutative group variety $\mathbb{Q}^{\ast}$ are included in rational points of commutative group variety $\mathbb{F}_{p}^{\ast}$ through modulo $p$. Note that cyclotomic polynomials over $\mathbb{Q}$ have even degree, they are out of scope of our target.\\

The main Theorem is as follow:
\begin{thm}[\bf{Reciprocity law of finite Galois extension field}]
Let $K$ be a finite algebraic field $([K:\mathbb{Q}]<\infty)$ and $\mathfrak{O}_{K}$ be its integer ring, $f(x) = x^{2g+1}+a_{1}x^{2g}+a_{2}x^{2g-1}+\cdots+a_{2g}x+a_{2g+1}$ be an odd degree irreducible polynomial over $\mathfrak{O}_{K}$ $(g>0)$, and $K_{f}$ be its minimal splitting field over $K$. For a prime ideal $\mathfrak{p}$ of $\mathfrak{O}_{K}$ which is relatively prime to 2 and conductor $\mathfrak{F} =\{\alpha\in \mathfrak{O}_{K(\theta)}|$ a principle ideal $(\alpha)$ of $\mathfrak{O}_{K(\theta)}\subset \mathfrak{O}_{K}[\theta]$, where $f(\theta)=0$ and $\theta\in K_{f}\}$, let $J_{f,\mathfrak{p}}$ be Jacobian Variety of a projective non-singular irreducible model $C_{f,\mathfrak{p}}$ of $y^2=f(x)$ of genus $g$ over $\mathfrak{O}_{K}/\mathfrak{p}$, where we assume that $f(x)=0$ has no multiple root over $\mathfrak{O}_{K}/\mathfrak{p}$. Then, the prime ideal $\mathfrak{p}$ of $\mathfrak{O}_{K}$ which is relatively prime to conductor $\mathfrak{F}$ splits completely over $K_{f}$ if and only if $J_{f, \mathfrak{p}}(2)$ $\cong$ $\mathbb{Z}/2\mathbb{Z}^{\oplus 2g}$.
\end{thm}

We show this Theorem in section 3 later. In section 2, we embed all roots of $f(x)=0$ into $2-$torsion subgroup $J_{f}(2)$ of Jacobian Variety of an irreducible non-singular projective curve model of $y^2=f(x)$. As a corollary, we have the following interesting result:

\begin{coro}
Let $f(x) = x^{2g+1}+a_{1}x^{2g}+a_{2}x^{2g-1}+\cdots+a_{2g}x+a_{2g+1}$ be an irreducible polynomial over a perfect field $F$ $(char(F)\neq 2)$, and $Gal(F_{f}/F)$ be its Galois group of minimal splitting field $F_{f}/F$ of $f(x)$. Then $Gal(F_{f}/F)$ is a subgroup of $GL_{2g}(\mathbb{F}_2)$.
\end{coro}

In section 3, by remembering Dedekind-Kummer's Theorem and "Inclusion Theorem" for algebraic number fields, we obtain our main Theorem 3.3 as reciprocity law of finite Galois extension fields defined by odd degree irreducible polynomials over algebraic integer ring.\\

\section{Embed all roots of $f(x)=0$ into a Jacobian Variety}
In this section, we consider a perfect field $F$ whose characteristic is not 2 because the following embedding theorem states over such fields.

\begin{thm}
For an irreducible polynomial $f(x) = x^{2g+1}+a_{1}x^{2g}+a_{2}x^{2g-1}+\cdots+a_{2g}x+a_{2g+1}$ over $F$, by considering a projective plane curve $y^2=f(x)$ and its non-singular projective model $C_{f}$ over $F$, we can embed all roots $e_{1}, e_{2}, \cdots, e_{2g+1}$ of $f(x)=0$ into $\overline{J}_{f}=J_{f}\otimes_{F}\overline{F}$ $(J_{f}$ is Jacobian Variety of $C_{f}$ and $\overline{F}$ is an algebraic closure of $F)$ as $P_{1}$ $-$ $P_{\infty}$, $P_{2}$ $-$ $P_{\infty}$, $\cdots$, $P_{2g+1}$ $-$ $P_{\infty}$, where  $P_{i}$ are non-singular points of $C_{f}$ corresponding to $[e_{i}:0:1]$ of $y^2=f(x)$ and $P_{\infty}$ is a rational point of $C_{f}$ over $F$ corresponding to $[0:1:0]$ of $y^2=f(x)$. Furthermore, the subgroup of $\overline{J}_{f}$ generated by the image of $P_{1}$ $-$ $P_{\infty}$, $P_{2}$ $-$ $P_{\infty}$, $\cdots$, $P_{2g+1}$ $-$ $P_{\infty}$ is $\overline{J}_{f}(2)$ $\cong$ $\mathbb{Z}/2\mathbb{Z}^{\oplus 2g}$.
\end{thm}

\begin{proof}
Let $[X:Y:Z]$ be the homogenous coordinate of $\mathbb{P}^{2}_{F}$ and consider the following homogenous equation of the target projective plane curve:\\
\begin{equation}
Z^{2g-1}Y^{2}=X^{2g+1}+a_{1}X^{2g}Z+a_{2}X^{2g-1}Z^2+\cdots+a_{2g}XZ^{2g}+a_{2g+1}Z^{2g+1}.
\end{equation}
Although infinity point $[0:1:0]$ is a singular point of $y^2=f(x)$ in projective plane $\mathbb{P}^{2}_{F}$ ($x=X/Z$ and $y=Y/Z$ in (2.1) above), we can resolve this singular point by repeating blow-up at $[0:1:0]$. Then please note that the inverse image of $[0:1:0]$ from projective singular curve $y^2=f(x)$ to $C_{f}$ consists of only one point because the degree of $f(x)$ is odd. Actually, the projective plane curve $y^2=f(x)$ is represented by the following equation near $[0:1:0]$ through $xz$ ($x=X/Y$ and $z=Z/Y$ in (2.1) above) plane by considering homogenous coordinate: \\
\begin{equation}
z^{2g-1}=x^{2g+1}+a_{1}x^{2g}z+a_{2}x^{2g-1}z^2+\cdots+a_{2g}xz^{2g}+a_{2g+1}z^{2g+1}.
\end{equation}
By putting $z=ux$, we can blow up at $(0, 0)$:\\
\begin{equation}
x^{2g-2}(x^2(1+a_{1}u+a_{2}u^2+\cdots +a_{2g}u^{2n}+a_{2g+1}u^{2g+1})-u^{2g-1})=0.
\end{equation}
Next, we focus on $x^2(1+a_{1}u+a_{2}u^2+\cdots +a_{2g}u^{2g}+a_{2g+1}u^{2g+1})-u^{2g-1}=0$. By putting $x=mu$, we obtain\\
\begin{equation}
u^2\{m^2(1+a_{1}u+a_{2}u^2+\cdots +a_{2g}u^{2g}+a_{2g+1}u^{2g+1})-u^{2g-3})=0.
\end{equation}
Therefore, by repeating blow-up at $(0, 0)$, we obtain the following curve near $(0, 0)$:
\begin{equation}
t^2(1+a_{1}u+a_{2}u^2+\cdots +a_{2g}u^{2g}+a_{2g+1}u^{2g+1})-u^{3}=0.
\end{equation}
According to $\cite{Fulton}$ Chapter 3 Problem 3.22, $(0,0)$ is a cusp and we can resolve this cusp by blowing up one more time, and the inverse image of $(0,0)$ consists of only one point. \\

If the degree of $f(x)$ is even and larger than 3, $[0:1:0]$ should be a crunode of $y^2=f(x)$, it can be approximated by a curve $xz=0$ near $[0:1:0]$. Therefore, the inverse image of $[0:1:0]$ from $y^2=f(x)$ to $C_{f}$ has at least 2 points, we cannot embed all roots of $f(x)=0$ into Jacobian Variety via the method above.\\

Next, we observe properties of $P{_i}-P_{\infty}$ on $\overline{J}_{f} = J_{f}\otimes_{F}\overline{F}$. Since div($x-e_{i}) = 2(P{_i}-P_{\infty})$ as a divisor of $\overline{C_{f}}=C_{f}\otimes_{F}\overline{F}$, $P{_i}-P_{\infty}$ is a $2$-torsion point on $J_{f}\otimes_{F}\overline{F}$. Furthermore, if $P_{i} - P_{\infty}$ $\sim$ $P_{j} - P_{\infty}$, $P_{i} $ $-$ $P_{j}$ $=$ div($h$) for some $h\in K(\overline{C_{f}})$ of function filed of $\overline{C_{f}}$. However this is impossible on $\overline{C_{f}}$ if $i\neq j$ because there is no element of $\overline{C_{f}}$'s function field such that order 1 at $P_{i}$, order -1 at ${P_j}$ and 0 at all other points. So we can embed all roots of $f(x)=0$ into $\overline{J}_{f}$ injectively. Since $div(y)=P_{1}+P_{2}+\cdots +P_{2g+1}-(2g+1)P_{\infty}=(P_{1}-P_{\infty}) + (P_{2}-P_{\infty}) + \cdots + (P_{2g+1}-P_{\infty})$ and the images of $P_{1}-P_{\infty}, P_{2}-P_{\infty}+\cdots, P_{2g-1}-P_{\infty}, P_{2g}-P_{\infty}$ on $\overline{J}_{f}(2)$ generate the subgroup $(\mathbb{Z}/2\mathbb{Z})^{\oplus 2g}$ of $\overline{J}_{f}(2)$. \\

Actually, assume that $\bar{a}_{1}(P_{1}-P_{\infty}) + \bar{a}_{2}(P_{2}-P_{\infty}) + \cdots + \bar{a}_{m}(P_{m}-P_{\infty}) = 0$ on $\overline{J}_{f}$ as a vector space over $\mathbb{F}_{2}$ for some $m<2g+1$, and $\bar{a}_{1}$, $\bar{a}_{2}$, $\cdots$ $\bar{a}_{m}\in \mathbb{F}_{2}$. Then, let $a_{1}P_{1} + a_{2}P_{2} + \cdots + a_{m}P_{m}$ be its inverse image in divisor group of $C$. On the other hand,  since $\sqrt{(x-e_{i})}$ is not included in $K(C_{f})$ and $m<2g+1$, $a_{i}$ is even because $div(x-e_{i}) = 2(P_{i}-P_{\infty}$). This means $\bar{a}_{i}=0$ for all $i$. Hence $P_{1}-P_{\infty}$, $P_{2}-P_{\infty}$, $\cdots$, and $P_{2g}-P_{\infty}$ on $\overline{J}_{f}$ are generators of $(\mathbb{Z}/2\mathbb{Z})^{\oplus 2g}$.

\end{proof}

\begin{coro}
Let $f(x) = x^{2g+1}+a_{1}x^{2g}+a_{2}x^{2g-1}+\cdots+a_{2g}x+a_{2g+1}$ be an irreducible polynomial over a perfect field $F$ $(char(F)\neq 2)$, and $Gal(F_{f}/F)$ be Galois group of minimal splitting field $F_{f}/F$ of $f(x)$. Then $Gal(F_{f}/F)$ is a subgroup of $GL_{2g}(\mathbb{F}_2)$.
\end{coro}
\begin{proof}
Since $Gal(F_{f}/F)$ induces permutation of roots of $f(x)=0$, it also induces basis change of $J_{f}\otimes_{F} {F_{f}}(2)\simeq  (\mathbb{Z}/2\mathbb{Z})^{\oplus2g}$. Furthermore, if the basis change of $J_{f}\otimes_{F} {F_{f}}(2)$ is trivial, the permutation of roots of $f(x)=0$ is also trivial.
\end{proof}

\section{Reciprocity Law}

Remember the following fundamental 2 theorems:
\begin{thm}[Dedekind-Kummer]
Let $K$ be an algebraic field $([K:\mathbb{Q}]<\infty)$ and $\theta$ be an algebraic integer, $L=K(\theta)$, $f(x)$ be its minimal polynomial of $\theta$ over $K$, and $\mathfrak{O}_{L}$ and $\mathfrak{O}_{K}$ be algebraic integer rings of $L$ and $K$ respectively, and $\overline{f}(x)$ $=$ $f(x)$ $(mod$ $\mathfrak{p})$. Then, for a prime ideal $\mathfrak{p}$ of $\mathfrak{O}_{K}$ which is relatively prime to conductor $\mathfrak{F} =\{\alpha\in \mathfrak{O}_{L}|$ a principle ideal $(\alpha)$ of $\mathfrak{O}_{L}\subset \mathfrak{O}_{K}[\theta]\}$, 
\begin{equation}
\mathfrak{p}\mathfrak{O}_{L}=\prod_{i=1}^{g}\mathfrak{P}_{i}^{e_i}
\end{equation}
where $\mathfrak{P}_{i}=\mathfrak{p}\mathfrak{O}_{L} + f_{i}(\theta)\mathfrak{O}_{L}$ is a prime ideal of $\mathfrak{O}_{L}$, $f(x)\equiv\prod_{i=1}^{g}\overline{f}_{i}(x)^{e_i}$ $(mod$ $\mathfrak{p})$ is irreducible decomposition of $f(x)$ $(mod$ $\mathfrak{p})$, and $f_{i}(x)$ is a pull back of $\overline{f_{i}}(x)$ to $\mathfrak{O}_{K}[x]$.
\end{thm}
\begin{proof}
See \cite{Jurgen} (8.3) Proposition.
\end{proof}

\begin{thm}[Inclusion Theorem, see \cite{Nancy} Chapter 5 Exercise 5.12.b]
Let $f(x), g(x)\in \frak{O}_{K}[x]$ be monic irreducible polynomials over algebraic integer ring $\frak{O}_{K}$ of an algebraic field $K$ $([K:\mathbb{Q}]<\infty)$ and $K_{f}$, $K_{g}$ be their minimal splitting fields. Then,

\begin{center}
$K_{f}\supset K_{g}$ if and only if $Spl(K_{f})$$\overset{\ast}{\subset}$$Spl(K_{g})$,
\end{center}
where $Spl(K_{f})$ $=$ $\{$prime ideal $\mathfrak{p}$ of $\mathfrak{O}_{K}$ such that $\mathfrak{p}$ splits completely over $K_{f}\}$ and $Spl(K_{f})$$\overset{\ast}{\subset}$$Spl(K_{g})$ means "$Spl(K_{f})$ is included in $Spl(K_{g})$ except for finite number of prime ideals".
\end{thm}
\begin{proof}
If $K_{f}\supset K_{g}$, then it is easy to see that $Spl(K_{f})$$\overset{\ast}{\subset}$$Spl(K_{g})$. Conversely, we assume that $Spl(K_{f})$$\overset{\ast}{\subset}$$Spl(K_{g})$. Then $Spl(K_{f}K_{g})$$\overset{\ast}{=}$$Spl(K_{f})$, where $A\overset{\ast}{=}B$ means $A=B$ except for finite number of elements for sets $A$ and $B$". Hence by Chebotarev's density theorem (see \cite{Cassels}, Chapter VIII), $[K_{f}K_{g}:K]=[K_{f}:K]$. This means that $K_{g}\subset K_{f}$.

\end{proof}

Now we consider an irreducible polynomial $f(x) = x^{2g+1}+a_{1}x^{2g}+a_{2}x^{2g-1}+\cdots+a_{2g}x+a_{2g+1}$ over algebraic integer ring $\mathfrak{O}_{K}$ of an algebraic field $K$ $([K:\mathbb{Q}]<\infty)$ and its minimal splitting filed $K_{f}=K(\alpha)$. If $g(x)$ is a minimal polynomial of $\alpha$ over $K$, $K_{f}=K_{g}$ and therefore $Spl(K_{f})$$\overset{\ast}=$$Spl(K_{g})$ ($\overset{\ast}=$ means "equal" except for finite number of prime numbers). 

Now, we consider $Spl(f)$ $=$ $\{$prime ideal $\mathfrak{p}$ of $\mathfrak{O}_{K}$ such that $f(x)$ mod $\mathfrak{p}$ is decomposed to distinct liner factors$\}$. Then, $Spl(f)\overset{\ast}{=}Spl(K_{f})$. Actually, if $\mathfrak{p}\in Spl(f)$, $\mathfrak{p}$ splits completely over $K_{f}$ because of Theorem 3.1.

Therefore, Theorem 3.1 and 3.2 lead that a prime ideal $\mathfrak{p}$ which is relatively prime to 2 and conductor $\mathfrak{F}$ splits completely over $K_{f}$ if and only if $f(x)$ is decomposed to a product of distinct linear factors over $\mathfrak{O}_{K}/\mathfrak{p}$. Furthermore, if $f(x)$ is decomposed to a product of distinct linear factors over $\mathfrak{O}_{K}/\mathfrak{p}$, $J_{f, \mathfrak{p}}(2)$ $\cong$ $\mathbb{Z}/2\mathbb{Z}^{\oplus 2g}$ because of Theorem 2.1, where $J_{f,\mathfrak{p}}$ is Jacobian Variety of a projective non-singular curve model $C_{f,\mathfrak{p}}$ of $y^2=f(x)$ over $\mathfrak{O}_{K}/\mathfrak{p}$ and $J_{f,\mathfrak{p}}(2)$ is 2-torsion subgroup of rational points of $J_{f, \mathfrak{p}}$ over $\mathfrak{O}_{K}/\mathfrak{p}$.\\

Therefore we obtain the following theorem:
\begin{thm}[\bf{Reciprocity law of finite Galois extension field}]
Let $K$ be a finite algebraic field $([K:\mathbb{Q}]<\infty)$ and $\mathfrak{O}_{K}$ be its integer ring, $f(x) = x^{2g+1}+a_{1}x^{2g}+a_{2}x^{2g-1}+\cdots+a_{2g}x+a_{2g+1}$ be an odd degree irreducible polynomial over $\mathfrak{O}_{K}$ $(g>0)$, and $K_{f}$ be its minimal splitting field over $K$. For a prime ideal $\mathfrak{p}$ of $\mathfrak{O}_{K}$ which is relatively prime to 2 and conductor $\mathfrak{F} =\{\alpha\in \mathfrak{O}_{K(\theta)}|$ a principle ideal $(\alpha)$ of $\mathfrak{O}_{K(\theta)}\subset \mathfrak{O}_{K}[\theta]$, where $f(\theta)=0$ and $\theta\in K_{f}\}$, let $J_{f,\mathfrak{p}}$ be Jacobian Variety of a projective non-singular irreducible model $C_{f,\mathfrak{p}}$ of $y^2=f(x)$ of genus $g$ over $\mathfrak{O}_{K}/\mathfrak{p}$, where we assume that $f(x)=0$ has no multiple root over $\mathfrak{O}_{K}/\mathfrak{p}$. Then, the prime ideal $\mathfrak{p}$ of $\mathfrak{O}_{K}$ which is relatively prime to conductor $\mathfrak{F}$ splits completely over $K_{f}$ if and only if $J_{f, \mathfrak{p}}(2)$ $\cong$ $\mathbb{Z}/2\mathbb{Z}^{\oplus 2g}$.
\end{thm}

Usual cyclotomic field theory uses commutative group variety $\mathbb{Q}^{\ast}$. All roots of $x^n=1$ are embedded into torsion subgroup of $\overline{\mathbb{Q}}^{\ast}$ naturally and a prime number $p$ splits completely over $\mathbb{Q}(\zeta_{n})$ if and only if $n\equiv 1$ (mod $p$), this means "all roots of $x^n=1$ over $\overline{\mathbb{F}}_{p}$ are included in rational points of commutative group variety $\mathbb{F}_{p}^{\ast}$". However, the degree of cyclotomic polynomials over $\mathbb{Q}$ is even. Therefore, Theorem 3.3 is essentially different from usual cyclotomic field theory.\\

%%%%%%%%%%%%%%%%%%%%%%%%%%%%%%%%%%%%%%%%%%%%%%%%%%%%%%%%%%%%%%%%%%%%%%%%%%%%%%%%%%%%%%%%%%%%%%%%%%%%%%%%%%%%%%%%%%
%%Reference start
%%%%%%%%%%%%%%%%%%%%%%%%%%%%%%%%%%%%%%%%%%%%%%%%%%%%%%%%%%%%%%%%%%%%%%%%%%%%%%%%%%%%%%%%%%%%%%%%%%%%%%%%%%%%%%%%%%

\end{document}